\documentclass[12pt,reqno]{article}

\usepackage[usenames]{color}

\usepackage[colorlinks=true,
linkcolor=webgreen,
filecolor=webbrown,
citecolor=webgreen]{hyperref}

\definecolor{webgreen}{rgb}{0,.5,0}
\definecolor{webbrown}{rgb}{.6,0,0}

\usepackage{fullpage}
\usepackage{float}

\usepackage{graphics,amsmath,amssymb}
\usepackage{graphicx}
\usepackage{amsthm}
\usepackage{amsfonts}
\usepackage{latexsym}
\usepackage{epsf}

\begin{document}

\theoremstyle{plain}
\newtheorem{theorem}{Theorem}
\newtheorem{corollary}[theorem]{Corollary}
\newtheorem{lemma}[theorem]{Lemma}
\newtheorem{proposition}[theorem]{Proposition}

\theoremstyle{definition}
\newtheorem{definition}[theorem]{Definition}
\newtheorem{example}[theorem]{Example}
\newtheorem{conjecture}[theorem]{Conjecture}

\theoremstyle{remark}
\newtheorem{remark}[theorem]{Remark}

\begin{center}
\vskip 1cm{\LARGE\bf Rapidly Convergent Summation Formulas\\
\vskip .1in involving Stirling Series}
\vskip 1cm
\large
Raphael Schumacher\\
MSc ETH Mathematics\\
Switzerland\\
\href{raphschu@ethz.ch}{\tt raphschu@ethz.ch}\\
\end{center}

\vskip .2 in

\begin{abstract}
This paper presents a family of rapidly convergent summation formulas for various finite sums of analytic functions. These summation formulas are obtained by applying a series acceleration transformation involving Stirling numbers of the first kind to the asymptotic, but divergent, expressions for the corresponding sums coming from the Euler-Maclaurin summation formula. While it is well-known that the expressions obtained from the Euler-Maclaurin summation formula diverge, our summation formulas are all very rapidly convergent and thus computationally efficient.
\end{abstract}

\section{Introduction}
\label{sec:introduction}

In this paper we will use asymptotic series representations for finite sums of various analytic functions, coming from the Euler-Maclaurin summation formula \cite{13,17}, to obtain rapidly convergent series expansions for these finite sums involving Stirling series \cite{1,2}. The key tool for doing this will be the so called Weniger transformation \cite{1} found by J. Weniger, which transforms an inverse power series into a Stirling series.\newline
\noindent For example, two of our three summation formulas for the sum $\sum_{k=0}^{n}\sqrt{k}$, where $n\in\mathbb{N}$, look like
\begin{displaymath}
\begin{split}
\sum_{k=0}^{n}\sqrt{k}&=\frac{2}{3}n^{\frac{3}{2}}+\frac{1}{2}\sqrt{n}-\frac{1}{4\pi}\zeta\left(\frac{3}{2}\right)+\sqrt{n}\sum_{k=1}^{\infty}(-1)^{k}\frac{\sum_{l=1}^{k}(-1)^{l}\frac{(2l-3)!!}{2^{l}(l+1)!}B_{l+1}S_{k}^{(1)}(l)}{(n+1)(n+2)\cdots(n+k)}\\
&=\frac{2}{3}n^{\frac{3}{2}}+\frac{1}{2}\sqrt{n}-\frac{1}{4\pi}\zeta\left(\frac{3}{2}\right)+\frac{\sqrt{n}}{24(n+1)}+\frac{\sqrt{n}}{24(n+1)(n+2)}+\frac{53\sqrt{n}}{640(n+1)(n+2)(n+3)}\\
&\quad+\frac{79\sqrt{n}}{320(n+1)(n+2)(n+3)(n+4)}+\ldots,
\end{split}
\end{displaymath}
\noindent or
\begin{displaymath}
\begin{split}
\sum_{k=0}^{n}\sqrt{k}&=\frac{2}{3}n^{\frac{3}{2}}+\frac{1}{2}\sqrt{n}-\frac{1}{4\pi}\zeta\left(\frac{3}{2}\right)+\frac{1}{24\sqrt{n}}+\sum_{k=1}^{\infty}(-1)^{k}\frac{\sum_{l=1}^{k}(-1)^{l}\frac{(2l-1)!!}{2^{l+1}(l+2)!}B_{l+2}S_{k}^{(1)}(l)}{\sqrt{n}(n+1)(n+2)\cdots(n+k)}\\
&=\frac{2}{3}n^{\frac{3}{2}}+\frac{1}{2}\sqrt{n}-\frac{1}{4\pi}\zeta\left(\frac{3}{2}\right)+\frac{1}{24\sqrt{n}}-\frac{1}{1920\sqrt{n}(n+1)(n+2)}-\frac{1}{640\sqrt{n}(n+1)(n+2)(n+3)}\\
&\quad-\frac{259}{46080\sqrt{n}(n+1)(n+2)(n+3)(n+4)}-\frac{115}{4608\sqrt{n}(n+1)(n+2)(n+3)(n+4)(n+5)}-\ldots,
\end{split}
\end{displaymath}
where the $B_{l}$'s are the Bernoulli numbers and $S_{k}^{(1)}(l)$ denotes the Stirling numbers of the first kind.
All other formulas in this article have a similar shape.\newline
\noindent We have searched our resulting formulas in the literature and on the internet, but we could find only three of them; therefore we believe that all others are new and our method to obtain them has not been recognized before.
The first of these three previously known formulas is
\begin{displaymath}
\begin{split}
\sum_{k=1}^{n}\frac{1}{k^{2}}&=\zeta(2)-\frac{1}{n}+\frac{1}{n^{2}}-\sum_{k=1}^{\infty}\frac{1}{k+1}\cdot\frac{k!}{n(n+1)(n+2)\cdots(n+k)}\\
&=\zeta(2)-\frac{1}{n}+\frac{1}{n^{2}}-\frac{1}{2n(n+1)}-\frac{2}{3n(n+1)(n+2)}-\frac{3}{2n(n+1)(n+2)(n+3)}\\
&\quad-\frac{24}{5n(n+1)(n+2)(n+3)(n+4)}-\ldots
\end{split}
\end{displaymath}
and was already known to Stirling around 1730 \cite{2}. This formula follows directly from the identity \cite{2}
\begin{displaymath}
\begin{split}
\sum_{k=n}^{\infty}\frac{1}{k^{2}}&=\frac{1}{n}+\sum_{k=1}^{\infty}\frac{1}{k+1}\cdot\frac{k!}{n(n+1)(n+2)\cdots(n+k)}\\
&=\frac{1}{n}+\frac{1}{2n(n+1)}+\frac{2}{3n(n+1)(n+2)}+\frac{3}{2n(n+1)(n+2)(n+3)}\\
&\quad+\frac{24}{5n(n+1)(n+2)(n+3)(n+4)}+\ldots.
\end{split}
\end{displaymath}
The second was obtained by Gregorio Fontana around 1780 \cite{3,4} and reads
\begin{displaymath}
\begin{split}
\sum_{k=1}^{n}\frac{1}{k}&=\log(n)+\gamma+\frac{1}{2n}-\sum_{k=1}^{\infty}\frac{\frac{(-1)^{k}}{k+1}\sum_{l=0}^{k+1}\frac{1}{l+1}S_{k+1}^{(1)}(l)}{n(n+1)(n+2)\cdots(n+k)}\\
&=\log(n)+\gamma+\frac{1}{2n}-\sum_{k=1}^{\infty}\frac{(-1)^{k}k!C_{k+1}}{n(n+1)(n+2)\cdots(n+k)}\\
&=\log(n)+\gamma+\frac{1}{2n}-\frac{1}{12n(n+1)}-\frac{1}{12n(n+1)(n+2)}-\frac{19}{120n(n+1)(n+2)(n+3)}\\
&\quad-\frac{9}{20n(n+1)(n+2)(n+3)(n+4)}-\ldots,
\end{split}
\end{displaymath}

\noindent where the \emph{Gregory numbers} $C_{k}$ for $k\in\mathbb{N}_{0}$ \cite{3,5,6,7} are defined by

\begin{displaymath}
\begin{split}
C_{k}:&=\frac{1}{k!}\int_{0}^{1}(x)_{\underline{k}}dx\\
&=\frac{1}{k!}\int_{0}^{1}x(x-1)\cdots(x-k+1)dx\\
&=\frac{1}{k!}\sum_{l=0}^{k}\frac{1}{l+1}S_{k}^{(1)}(l).
\end{split}
\end{displaymath}
The last one is a convergent version of Stirling's formula for the logarithm of the factorial function \cite{8}, namely
\begin{displaymath}
\begin{split}
\sum_{k=1}^{n}\log(k)&=n\log(n)-n+\frac{1}{2}\log\left(2\pi\right)+\frac{1}{2}\log(n)+\sum_{k=1}^{\infty}\frac{\frac{(-1)^{k}}{2k}\sum_{l=1}^{k}\frac{(-1)^{l}l}{(l+1)(l+2)}S_{k}^{(1)}(l)}{(n+1)(n+2)\cdots(n+k)}\\
&=n\log(n)-n+\frac{1}{2}\log\left(2\pi\right)+\frac{1}{2}\log(n)+\sum_{k=1}^{\infty}\frac{a_{k}}{(n+1)(n+2)\cdots(n+k)}\\
&=n\log(n)-n+\frac{1}{2}\log\left(2\pi\right)+\frac{1}{2}\log(n)+\frac{1}{12(n+1)}+\frac{1}{12(n+1)(n+2)}\\
&\quad+\frac{59}{360(n+1)(n+2)(n+3)}+\frac{29}{60(n+1)(n+2)(n+3)(n+4)}+\ldots,
\end{split}
\end{displaymath}
where the numbers $a_{k}$ for $k\in\mathbb{N}$ \cite{8,9,10} are defined by
\begin{displaymath}
\begin{split}
a_{k}:&=\frac{1}{k}\int_{0}^{1}(x)_{k}\left(x-\frac{1}{2}\right)dx\\
&=\frac{1}{2k}\sum_{l=1}^{k}\frac{l}{(l+1)(l+2)}\left|S_{k}^{(1)}(l)\right|\\
&=\frac{(-1)^{k}}{2k}\sum_{l=1}^{k}\frac{(-1)^{l}l}{(l+1)(l+2)}S_{k}^{(1)}(l).
\end{split}
\end{displaymath}
and was also found around 1800.

\section{Definitions}
\label{sec:Definitions}

\begin{definition}(Fractional part $\{x\}$ of $x$)\newline
\noindent Let $\lfloor x\rfloor$ denote the \emph{floor function} evaluated at the point $x\in\mathbb{R}$, which is the largest integer less than or equal to $x$. Written out explicitly, this means that
\begin{displaymath}
\lfloor x\rfloor:=\max_{m\in\mathbb{Z}}\left\{m\leq x\right\}.
\end{displaymath}
We define by $\{x\}$ the \emph{fractional part} of $x$, that is
\begin{displaymath}
\{x\}:=x-\lfloor x\rfloor.
\end{displaymath}
Note that we have $0\leq\{x\}<1$.
\end{definition}

\begin{definition}(Pochhammer symbol)\cite{1}\newline
\noindent We define the \emph{Pochhammer symbol} (or rising factorial function) $(x)_{k}$ by 
\begin{displaymath}
(x)_{k}:=x(x+1)(x+2)(x+3)\cdots(x+k-1)=\frac{\Gamma(x+k)}{\Gamma(x)},
\end{displaymath}
where $\Gamma(x)$ is the gamma function defined by
\begin{displaymath}
\Gamma(x):=\int_{0}^{\infty}e^{-t}t^{x-1}dt.
\end{displaymath}
\end{definition}

\begin{definition}(Stirling Series)\cite{1,2}\newline
\noindent A \emph{Stirling series} (or inverse factorial series) for a function $f:\mathbb{C}\rightarrow\mathbb{C}$, which vanishes at $z\rightarrow+\infty$, is an expansion of the type
\begin{displaymath}
f(x)=\sum_{k=0}^{\infty}\frac{k!a_{k}}{(x)_{k+1}}=\frac{0!a_{0}}{x}+\frac{1!a_{1}}{x(x+1)}+\frac{2!a_{2}}{x(x+1)(x+2)}+\frac{3!a_{3}}{x(x+1)(x+2)(x+3)}+\ldots,
\end{displaymath}
where $(x)_{k}$ denotes the Pochhammer symbol.
\end{definition}

\begin{definition}(Stirling numbers of the first kind)\cite{1}\newline
\noindent Let $k,l\in\mathbb{N}_{0}$ be two non-negative integers such that $k\geq l\geq0$. We define the \emph{Stirling numbers of the first kind} $S_{k}^{(1)}(l)$ as the connecting coefficients in the identity
\begin{displaymath}
(x)_{k}=(-1)^{k}\sum_{l=0}^{k}(-1)^{l}S_{k}^{(1)}(l)x^{l},
\end{displaymath}
where $(x)_{k}$ is the rising factorial.
\end{definition}

\begin{definition}(Bernoulli numbers)\cite{12}\newline
\noindent We define the $k$-th \emph{Bernoulli number} $B_{k}$ as the $k$-th coefficient in the generating function relation
\begin{displaymath}
\frac{x}{e^{x}-1}=\sum_{k=0}^{\infty}\frac{B_{k}}{k!}x^{k}\;\;\forall x\in\mathbb{C}\;\text{with $|x|<2\pi$}.
\end{displaymath}
\end{definition}

\begin{definition}(Euler numbers)\cite{13}\newline
\noindent We define the sequence of \emph{Euler numbers} $\left\{E_{k}\right\}_{k=0}^{\infty}$ by the generating function identity
\begin{displaymath}
\frac{2e^{x}}{e^{2x}+1}=\sum_{k=0}^{\infty}\frac{E_{k}}{k!}x^{k}.
\end{displaymath}
\end{definition}

\begin{definition}(Tangent numbers)\cite{14,15}\newline
\noindent We define the $k$-th \emph{tangent number} $T_{k}$ by the equation
\begin{displaymath}
T_{k}:=\frac{2^{2k}(2^{2k}-1)|B_{2k}|}{2k},
\end{displaymath}
where the $B_{k}$'s are the Bernoulli numbers.
\end{definition}

\begin{definition}(Bernoulli polynomials)\cite[\text{Proposition\ 23.2, p.\ 86}]{16}\newline
\noindent We define the $n$-th \emph{Bernoulli polynomial} $B_{n}(x)$ by
\begin{displaymath}
B_{n}(x):=\sum_{k=0}^{n}{n\choose k}B_{k}x^{n-k},
\end{displaymath}
where the $B_{k}$'s are the Bernoulli numbers.
\end{definition}

\section{The Summation Formulas}
\label{sec:The Summation Formulas}

In this section we will list $32$ summation formulas for various finite sums of analytic functions. 
For this, we need the following 

\begin{lemma}(Euler-Maclaurin summation formula)\cite{13,17}\newline
\noindent Let $f$ be an analytic function. Then we have that
\begin{displaymath}
\begin{split}
\sum_{k=1}^{n}f(k)&=\int_{1}^{n}f(x)dx+\frac{1}{2}\left(f(n)+f(1)\right)+\sum_{k=2}^{m}\frac{B_{k}}{k!}\left(f^{(k-1)}(n)-f^{(k-1)}(1)\right)\\
&\quad+\frac{(-1)^{m+1}}{m!}\int_{1}^{n}B_{m}(\{x\})f^{(m)}(x)dx,
\end{split}
\end{displaymath}
where $B_{m}(x)$ is the $m$-th Bernoulli polynomial and $\{x\}$ denotes the fractional part of $x$.
Therefore, for many functions $f$, we have the asymptotic expansion
\begin{displaymath}
\begin{split}
\sum_{k=1}^{n}f(k)\sim\int_{1}^{n}f(x)dx+C-\frac{B_{1}}{1!}f(n)+\sum_{k=2}^{\infty}\frac{B_{k}}{k!}f^{(k-1)}(n)\;\;\text{as $n\rightarrow\infty$},
\end{split}
\end{displaymath}
for some constant $C\in\mathbb{C}$.
\end{lemma}
\noindent As well as the next key result found by J. Weniger:

\begin{lemma}(Weniger transformation)\cite{1}\newline
\noindent For every inverse power series $\sum_{k=1}^{\infty}\frac{a_{k}}{x^{k+1}}$, the following transformation formula holds
\begin{displaymath}
\begin{split}
\sum_{k=1}^{\infty}\frac{a_{k}}{x^{k+1}}&=\sum_{k=1}^{\infty}\frac{(-1)^{k}}{(x)_{k+1}}\sum_{l=1}^{k}(-1)^{l}a_{l}S_{k}^{(1)}(l)\\
&=\sum_{k=1}^{\infty}\frac{(-1)^{k}\sum_{l=1}^{k}(-1)^{l}a_{l}S_{k}^{(1)}(l)}{x(x+1)(x+2)\cdots(x+k)}.
\end{split}
\end{displaymath}
\end{lemma}

Now it is time for our summation formulas. For every single finite sum there are infinitely many different such formulas, therefore we give here just the canonical ones. It is not hard to prove, using the asymptotic expansions of the Bernoulli numbers and the Stirling numbers of the first kind as well as that of the Euler numbers, that all series below converges very rapidly.

\begin{theorem}(Summation formulas for the harmonic series)\newline
\noindent For every natural number $n\in\mathbb{N}$, we have that
\begin{displaymath}
\begin{split}
\sum_{k=1}^{n}\frac{1}{k}&=\log(n)+\gamma+\frac{1}{2n}+\sum_{k=1}^{\infty}\frac{(-1)^{k+1}\sum_{l=1}^{k}\frac{(-1)^{l}}{l+1}B_{l+1}S_{k}^{(1)}(l)}{n(n+1)(n+2)\cdots(n+k)}\\
&=\log(n)+\gamma+\frac{1}{2n}-\frac{1}{12n(n+1)}-\frac{1}{12n(n+1)(n+2)}-\frac{19}{120n(n+1)(n+2)(n+3)}\\
&\quad-\frac{9}{20n(n+1)(n+2)(n+3)(n+4)}-\ldots.
\end{split}
\end{displaymath}
\noindent We also have that
\begin{displaymath}
\begin{split}
\sum_{k=1}^{n}\frac{1}{k}&=\log(n)+\gamma+\sum_{k=1}^{\infty}\frac{(-1)^{k+1}\sum_{l=1}^{k}\frac{(-1)^{l}}{l}B_{l}S_{k}^{(1)}(l)}{(n+1)(n+2)\cdots(n+k)}\\
&=\log(n)+\gamma+\frac{1}{2(n+1)}+\frac{5}{12(n+1)(n+2)}+\frac{3}{4(n+1)(n+2)(n+3)}\\
&\quad+\frac{251}{120(n+1)(n+2)(n+3)(n+4)}+\ldots.
\end{split}
\end{displaymath}
\end{theorem}

\begin{proof}
Applying the Euler-Maclaurin summation formula to the function $f(x):=\frac{1}{x}$, we get that
\begin{displaymath}
\sum_{k=1}^{n}\frac{1}{k}\sim\log(n)+\gamma-\sum_{k=1}^{\infty}\frac{B_{k}}{kn^{k}}.
\end{displaymath}
Applying now the Weniger transformation to the equivalent series
\begin{displaymath}
\begin{split}
\sum_{k=1}^{n}\frac{1}{k}&\sim\log(n)+\gamma+\frac{1}{2n}-\sum_{k=2}^{\infty}\frac{B_{k}}{kn^{k}}\\
&\sim\log(n)+\gamma+\frac{1}{2n}-\sum_{k=1}^{\infty}\frac{B_{k+1}}{(k+1)n^{k+1}},
\end{split}
\end{displaymath}
we get the first claimed formula.
To obtain the second expression, we apply the Weniger transformation 
to the identity
\begin{displaymath}
\sum_{k=1}^{n}\frac{1}{k}\sim\log(n)+\gamma-n\sum_{k=1}^{\infty}\frac{B_{k}}{kn^{k+1}}.
\end{displaymath}
\end{proof}

\noindent All summation formulas in this article, which we list below, are obtained using the same universal technique: For each finite sum, we used first the Euler-Maclaurin summation formula and then applied to it the Weniger transformation to get convergent formulas. Therefore, we omit full proofs and list here just our final results.

\begin{itemize}
\item[1.)]{{\bf Summation formulas for the harmonic series:}\newline
\noindent For every natural number $n\in\mathbb{N}$, we have
\begin{displaymath}
\begin{split}
\sum_{k=1}^{n}\frac{1}{k}&=\log(n)+\gamma+\frac{1}{2n}+\sum_{k=1}^{\infty}\frac{(-1)^{k+1}\sum_{l=1}^{k}\frac{(-1)^{l}}{l+1}B_{l+1}S_{k}^{(1)}(l)}{n(n+1)(n+2)\cdots(n+k)}\\
&=\log(n)+\gamma+\frac{1}{2n}-\frac{1}{12n(n+1)}-\frac{1}{12n(n+1)(n+2)}-\frac{19}{120n(n+1)(n+2)(n+3)}\\
&\quad-\frac{9}{20n(n+1)(n+2)(n+3)(n+4)}-\ldots
\end{split}
\end{displaymath}
\noindent and
\begin{displaymath}
\begin{split}
\sum_{k=1}^{n}\frac{1}{k}&=\log(n)+\gamma+\sum_{k=1}^{\infty}\frac{(-1)^{k+1}\sum_{l=1}^{k}\frac{(-1)^{l}}{l}B_{l}S_{k}^{(1)}(l)}{(n+1)(n+2)\cdots(n+k)}\\
&=\log(n)+\gamma+\frac{1}{2(n+1)}+\frac{5}{12(n+1)(n+2)}+\frac{3}{4(n+1)(n+2)(n+3)}\\
&\quad+\frac{251}{120(n+1)(n+2)(n+3)(n+4)}+\ldots.
\end{split}
\end{displaymath}
}

\item[2.)]{{\bf Summation formulas for the $n$-th partial sum of $\zeta(2)$:}\newline
\noindent For every natural number $n\in\mathbb{N}$, we have
\begin{displaymath}
\begin{split}
\sum_{k=1}^{n}\frac{1}{k^{2}}&=\zeta(2)-\frac{1}{n}+\frac{1}{2n^{2}}+\frac{1}{n}\sum_{k=1}^{\infty}\frac{(-1)^{k+1}\sum_{l=1}^{k}(-1)^{l}B_{l+1}S_{k}^{(1)}(l)}{n(n+1)(n+2)\cdots(n+k)}\\
&=\zeta(2)-\frac{1}{n}+\frac{1}{2n^{2}}-\frac{1}{6n^{2}(n+1)}-\frac{1}{6n^{2}(n+1)(n+2)}-\frac{3}{10n^{2}(n+1)(n+2)(n+3)}\\
&\quad-\frac{4}{5n^{2}(n+1)(n+2)(n+3)(n+4)}-\ldots
\end{split}
\end{displaymath}
and
\begin{displaymath}
\begin{split}
\sum_{k=1}^{n}\frac{1}{k^{2}}&=\zeta(2)-\frac{1}{n}+\sum_{k=1}^{\infty}\frac{(-1)^{k+1}\sum_{l=1}^{k}(-1)^{l}B_{l}S_{k}^{(1)}(l)}{n(n+1)(n+2)\cdots(n+k)}\\
&=\zeta(2)-\frac{1}{n}+\sum_{k=1}^{\infty}\frac{1}{k+1}\cdot\frac{(k-1)!}{n(n+1)(n+2)\cdots(n+k)}\\
&=\zeta(2)-\frac{1}{n}+\frac{1}{2n(n+1)}+\frac{1}{3n(n+1)(n+2)}+\frac{1}{2n(n+1)(n+2)(n+3)}\\
&\quad+\frac{6}{5n(n+1)(n+2)(n+3)(n+4)}+\ldots.
\end{split}
\end{displaymath}}

\item[3.)]{{\bf Summation formulas for the $n$-th partial sum of $\zeta(3)$:}\newline
\noindent For every natural number $n\in\mathbb{N}$, we have\begin{displaymath}
\begin{split}
\sum_{k=1}^{n}\frac{1}{k^{3}}&=\zeta(3)-\frac{1}{2n^{2}}+\frac{1}{2n^{3}}+\frac{1}{2n^{2}}\sum_{k=1}^{\infty}\frac{(-1)^{k+1}\sum_{l=1}^{k}(-1)^{l}(l+2)B_{l+1}S_{k}^{(1)}(l)}{n(n+1)(n+2)\cdots(n+k)}\\
&=\zeta(3)-\frac{1}{2n^{2}}+\frac{1}{2n^{3}}-\frac{1}{4n^{3}(n+1)}-\frac{1}{4n^{3}(n+1)(n+2)}-\frac{5}{12n^{3}(n+1)(n+2)(n+3)}\\
&\quad-\frac{1}{n^{3}(n+1)(n+2)(n+3)(n+4)}-\ldots
\end{split}
\end{displaymath}
and
\begin{displaymath}
\begin{split}
\sum_{k=1}^{n}\frac{1}{k^{3}}&=\zeta(3)-\frac{1}{2n^{2}}+\frac{1}{2}\sum_{k=1}^{\infty}\frac{(-1)^{k+1}\sum_{l=1}^{k}(-1)^{l}(l+1)B_{l}S_{k}^{(1)}(l)}{n^{2}(n+1)(n+2)\cdots(n+k)}\\
&=\zeta(3)-\frac{1}{2n^{2}}+\frac{1}{2n^{2}(n+1)}+\frac{1}{4n^{2}(n+1)(n+2)}+\frac{1}{4n^{2}(n+1)(n+2)(n+3)}\\
&\quad+\frac{1}{3n^{2}(n+1)(n+2)(n+3)(n+4)}+\ldots.
\end{split}
\end{displaymath}}

\item[4.)]{{\bf Summation formulas for the sum of the square roots of the first $n$ natural numbers:}\newline
\noindent For every natural number $n\in\mathbb{N}$, we have\begin{displaymath}
\begin{split}
\sum_{k=0}^{n}\sqrt{k}&=\frac{2}{3}n^{\frac{3}{2}}+\frac{1}{2}\sqrt{n}-\frac{1}{4\pi}\zeta\left(\frac{3}{2}\right)+\sqrt{n}\sum_{k=1}^{\infty}(-1)^{k}\frac{\sum_{l=1}^{k}(-1)^{l}\frac{(2l-3)!!}{2^{l}(l+1)!}B_{l+1}S_{k}^{(1)}(l)}{(n+1)(n+2)\cdots(n+k)}\\
&=\frac{2}{3}n^{\frac{3}{2}}+\frac{1}{2}\sqrt{n}-\frac{1}{4\pi}\zeta\left(\frac{3}{2}\right)+\frac{\sqrt{n}}{24(n+1)}+\frac{\sqrt{n}}{24(n+1)(n+2)}+\frac{53\sqrt{n}}{640(n+1)(n+2)(n+3)}\\
&\quad+\frac{79\sqrt{n}}{320(n+1)(n+2)(n+3)(n+4)}+\ldots,
\end{split}
\end{displaymath}
as well as
\begin{displaymath}
\begin{split}
\sum_{k=0}^{n}\sqrt{k}&=\frac{2}{3}n^{\frac{3}{2}}+\frac{1}{2}\sqrt{n}-\frac{1}{4\pi}\zeta\left(\frac{3}{2}\right)+\frac{1}{24\sqrt{n}}+\sum_{k=1}^{\infty}(-1)^{k}\frac{\sum_{l=1}^{k}(-1)^{l}\frac{(2l-1)!!}{2^{l+1}(l+2)!}B_{l+2}S_{k}^{(1)}(l)}{\sqrt{n}(n+1)(n+2)\cdots(n+k)}\\
&=\frac{2}{3}n^{\frac{3}{2}}+\frac{1}{2}\sqrt{n}-\frac{1}{4\pi}\zeta\left(\frac{3}{2}\right)+\frac{1}{24\sqrt{n}}-\frac{1}{1920\sqrt{n}(n+1)(n+2)}\\
&\quad-\frac{1}{640\sqrt{n}(n+1)(n+2)(n+3)}-\frac{259}{46080\sqrt{n}(n+1)(n+2)(n+3)(n+4)}\\
&\quad-\frac{115}{4608\sqrt{n}(n+1)(n+2)(n+3)(n+4)(n+5)}-\ldots
\end{split}
\end{displaymath}
and
\begin{displaymath}
\begin{split}
\sum_{k=0}^{n}\sqrt{k}&=\frac{2}{3}n^{\frac{3}{2}}-\frac{1}{4\pi}\zeta\left(\frac{3}{2}\right)+n\sqrt{n}\sum_{k=1}^{\infty}(-1)^{k}\frac{\sum_{l=1}^{k}(-1)^{l}\frac{(2l-5)!!}{2^{l-1}l!}B_{l}S_{k}^{(1)}(l)}{(n+1)(n+2)\cdots(n+k)}\\
&=\frac{2}{3}n^{\frac{3}{2}}-\frac{1}{4\pi}\zeta\left(\frac{3}{2}\right)+\frac{n\sqrt{n}}{2(n+1)}+\frac{13n\sqrt{n}}{24(n+1)(n+2)}+\frac{9n\sqrt{n}}{8(n+1)(n+2)(n+3)}\\
&\quad+\frac{2213n\sqrt{n}}{640(n+1)(n+2)(n+3)(n+4)}+\ldots.
\end{split}
\end{displaymath}
}

\item[5.)]{{\bf Summation formulas for the $n$-th partial sum of $\zeta(-3/2)$:}\newline
\noindent For every natural number $n\in\mathbb{N}$, we have
\begin{displaymath}
\begin{split}
\sum_{k=0}^{n}k\sqrt{k}&=\frac{2}{5}n^{\frac{5}{2}}+\frac{1}{2}n^{\frac{3}{2}}-\frac{3}{16\pi^{2}}\zeta\left(\frac{5}{2}\right)+\frac{3}{2}n^{\frac{3}{2}}\sum_{k=1}^{\infty}(-1)^{k+1}\frac{\sum_{l=1}^{k}(-1)^{l}\frac{(2l-5)!!}{2^{l-1}(l+1)!}B_{l+1}S_{k}^{(1)}(l)}{(n+1)(n+2)\cdots(n+k)}\\
&=\frac{2}{5}n^{\frac{5}{2}}+\frac{1}{2}n^{\frac{3}{2}}-\frac{3}{16\pi^{2}}\zeta\left(\frac{5}{2}\right)+\frac{n\sqrt{n}}{8(n+1)}+\frac{n\sqrt{n}}{8(n+1)(n+2)}+\frac{481n\sqrt{n}}{1920(n+1)(n+2)(n+3)}\\
&\quad+\frac{241n\sqrt{n}}{320(n+1)(n+2)(n+3)(n+4)}+\ldots,
\end{split}
\end{displaymath}
as well as
\begin{displaymath}
\begin{split}
\sum_{k=0}^{n}k\sqrt{k}&=\frac{2}{5}n^{\frac{5}{2}}+\frac{1}{2}n^{\frac{3}{2}}+\frac{1}{8}\sqrt{n}-\frac{3}{16\pi^{2}}\zeta\left(\frac{5}{2}\right)+\frac{3}{2}\sum_{k=1}^{\infty}(-1)^{k+1}\frac{\sum_{l=1}^{k}(-1)^{l}\frac{(2l-1)!!}{2^{l+1}(l+3)!}B_{l+3}S_{k}^{(1)}(l)}{\sqrt{n}(n+1)(n+2)\cdots(n+k)}\\
&=\frac{2}{5}n^{\frac{5}{2}}+\frac{1}{2}n^{\frac{3}{2}}+\frac{1}{8}\sqrt{n}-\frac{3}{16\pi^{2}}\zeta\left(\frac{5}{2}\right)+\frac{1}{1920\sqrt{n}(n+1)}+\frac{1}{1920\sqrt{n}(n+1)(n+2)}\\
&\quad+\frac{107}{107520\sqrt{n}(n+1)(n+2)(n+3)}+\frac{51}{17920\sqrt{n}(n+1)(n+2)(n+3)(n+4)}+\ldots
\end{split}
\end{displaymath}
and
\begin{displaymath}
\begin{split}
\sum_{k=0}^{n}k\sqrt{k}&=\frac{2}{5}n^{\frac{5}{2}}-\frac{3}{16\pi^{2}}\zeta\left(\frac{5}{2}\right)+\frac{3}{2}n^{\frac{5}{2}}\sum_{k=1}^{\infty}(-1)^{k+1}\frac{\sum_{l=1}^{k}(-1)^{l}\frac{(2l-7)!!}{2^{l-2}l!}B_{l}S_{k}^{(1)}(l)}{(n+1)(n+2)\cdots(n+k)}\\
&=\frac{2}{5}n^{\frac{5}{2}}-\frac{3}{16\pi^{2}}\zeta\left(\frac{5}{2}\right)+\frac{n^{2}\sqrt{n}}{2(n+1)}+\frac{5n^{2}\sqrt{n}}{8(n+1)(n+2)}+\frac{11n^{2}\sqrt{n}}{8(n+1)(n+2)(n+3)}\\
&\quad+\frac{8401n^{2}\sqrt{n}}{1920(n+1)(n+2)(n+3)(n+4)}+\ldots.
\end{split}
\end{displaymath}}

\item[6.)]{{\bf Summation formulas for the $n$-th partial sum of $\zeta(-5/2)$:}\newline
\noindent For every natural number $n\in\mathbb{N}$, we have
\begin{displaymath}
\begin{split}
\sum_{k=0}^{n}k^{2}\sqrt{k}&=\frac{2}{7}n^{\frac{7}{2}}+\frac{1}{2}n^{\frac{5}{2}}+\frac{15}{64\pi^{3}}\zeta\left(\frac{7}{2}\right)+\frac{15}{4}n^{\frac{5}{2}}\sum_{k=1}^{\infty}(-1)^{k}\frac{\sum_{l=1}^{k}(-1)^{l}\frac{(2l-7)!!}{2^{l-2}(l+1)!}B_{l+1}S_{k}^{(1)}(l)}{(n+1)(n+2)\cdots(n+k)}\\
&=\frac{2}{7}n^{\frac{7}{2}}+\frac{1}{2}n^{\frac{5}{2}}+\frac{15}{64\pi^{3}}\zeta\left(\frac{7}{2}\right)+\frac{5n^{2}\sqrt{n}}{24(n+1)}+\frac{5n^{2}\sqrt{n}}{24(n+1)(n+2)}\\
&\quad+\frac{53n^{2}\sqrt{n}}{128(n+1)(n+2)(n+3)}+\frac{79n^{2}\sqrt{n}}{64(n+1)(n+2)(n+3)(n+4)}\\
&\quad+\frac{35187n^{2}\sqrt{n}}{7168(n+1)(n+2)(n+3)(n+4)(n+5)}+\ldots,
\end{split}
\end{displaymath}
as well as
\begin{displaymath}
\begin{split}
\sum_{k=0}^{n}k^{2}\sqrt{k}&=\frac{2}{7}n^{\frac{7}{2}}+\frac{1}{2}n^{\frac{5}{2}}+\frac{5}{24}n^{\frac{3}{2}}+\frac{15}{64\pi^{3}}\zeta\left(\frac{7}{2}\right)-\frac{1}{384\sqrt{n}}\\
&\quad+\frac{15}{4}\sum_{k=1}^{\infty}(-1)^{k}\frac{\sum_{l=1}^{k}(-1)^{l}\frac{(2l-1)!!}{2^{l+1}(l+4)!}B_{l+4}S_{k}^{(1)}(l)}{\sqrt{n}(n+1)(n+2)\cdots(n+k)}\\
&=\frac{2}{7}n^{\frac{7}{2}}+\frac{1}{2}n^{\frac{5}{2}}+\frac{5}{24}n^{\frac{3}{2}}+\frac{15}{64\pi^{3}}\zeta\left(\frac{7}{2}\right)-\frac{1}{384\sqrt{n}}+\frac{1}{21504\sqrt{n}(n+1)(n+2)}\\
&\quad+\frac{1}{7168\sqrt{n}(n+1)(n+2)(n+3)}+\frac{115}{229376\sqrt{n}(n+1)(n+2)(n+3)(n+4)}\\
&\quad+\frac{255}{114688\sqrt{n}(n+1)(n+2)(n+3)(n+4)(n+5)}+\ldots
\end{split}
\end{displaymath}
and
\begin{displaymath}
\begin{split}
\sum_{k=0}^{n}k^{2}\sqrt{k}&=\frac{2}{7}n^{\frac{7}{2}}+\frac{15}{64\pi^{3}}\zeta\left(\frac{7}{2}\right)+\frac{15}{4}n^{\frac{7}{2}}\sum_{k=1}^{\infty}(-1)^{k}\frac{\sum_{l=1}^{k}(-1)^{l}\frac{(2l-9)!!}{2^{l-3}l!}B_{l}S_{k}^{(1)}(l)}{(n+1)(n+2)\cdots(n+k)}\\
&=\frac{2}{7}n^{\frac{7}{2}}+\frac{15}{64\pi^{3}}\zeta\left(\frac{7}{2}\right)+\frac{n^{3}\sqrt{n}}{2(n+1)}+\frac{17n^{3}\sqrt{n}}{24(n+1)(n+2)}+\frac{13n^{3}\sqrt{n}}{8(n+1)(n+2)(n+3)}\\
&\quad+\frac{677n^{3}\sqrt{n}}{128(n+1)(n+2)(n+3)(n+4)}+\ldots.
\end{split}
\end{displaymath}}

\item[7.)]{{\bf Summation formulas for the sum of the inverses of the square roots of the first $n$ natural numbers:}\newline
\noindent For every natural number $n\in\mathbb{N}$, we have
\begin{displaymath}
\begin{split}
\sum_{k=1}^{n}\frac{1}{\sqrt{k}}
&=2\sqrt{n}+\zeta\left(\frac{1}{2}\right)+\frac{1}{2\sqrt{n}}+\sum_{k=1}^{\infty}(-1)^{k+1}\frac{\sum_{l=1}^{k}(-1)^{l}\frac{(2l-1)!!}{2^{l}(l+1)!}B_{l+1}S_{k}^{(1)}(l)}{\sqrt{n}(n+1)(n+2)\cdots(n+k)}\\
&=2\sqrt{n}+\zeta\left(\frac{1}{2}\right)+\frac{1}{2\sqrt{n}}-\frac{1}{24\sqrt{n}(n+1)}-\frac{1}{24\sqrt{n}(n+1)(n+2)}\\
&\quad-\frac{31}{384\sqrt{n}(n+1)(n+2)(n+3)}-\frac{15}{64\sqrt{n}(n+1)(n+2)(n+3)(n+4)}-\ldots
\end{split}
\end{displaymath}
and
\begin{displaymath}
\begin{split}
\sum_{k=1}^{n}\frac{1}{\sqrt{k}}
&=2\sqrt{n}+\zeta\left(\frac{1}{2}\right)+\sqrt{n}\sum_{k=1}^{\infty}(-1)^{k+1}\frac{\sum_{l=1}^{k}(-1)^{l}\frac{(2l-3)!!}{2^{l-1}l!}B_{l}S_{k}^{(1)}(l)}{(n+1)(n+2)\cdots(n+k)}\\
&=2\sqrt{n}+\zeta\left(\frac{1}{2}\right)+\frac{\sqrt{n}}{2(n+1)}+\frac{11\sqrt{n}}{24(n+1)(n+2)}+\frac{7\sqrt{n}}{8(n+1)(n+2)(n+3)}\\
&\quad+\frac{977\sqrt{n}}{384(n+1)(n+2)(n+3)(n+4)}+\ldots.
\end{split}
\end{displaymath}
}

\item[8.)]{{\bf Summation formulas for the $n$-th partial sum of $\zeta(3/2)$:}\newline
\noindent For every natural number $n\in\mathbb{N}$, we have
\begin{displaymath}
\begin{split}
\sum_{k=1}^{n}\frac{1}{k\sqrt{k}}
&=\zeta\left(\frac{3}{2}\right)-\frac{2}{\sqrt{n}}+\frac{1}{2n\sqrt{n}}+\frac{2}{\sqrt{n}}\sum_{k=1}^{\infty}\frac{(-1)^{k+1}\sum_{l=1}^{k}(-1)^{l}\frac{(2l+1)!!}{2^{l+1}(l+1)!}B_{l+1}S_{k}^{(1)}(l)}{n(n+1)(n+2)\cdots(n+k)}\\
&=\zeta\left(\frac{3}{2}\right)-\frac{2}{\sqrt{n}}+\frac{1}{2n\sqrt{n}}-\frac{1}{8n\sqrt{n}(n+1)}-\frac{1}{8n\sqrt{n}(n+1)(n+2)}\\
&\quad-\frac{89}{384n\sqrt{n}(n+1)(n+2)(n+3)}-\frac{41}{64n\sqrt{n}(n+1)(n+2)(n+3)(n+4)}-\ldots
\end{split}
\end{displaymath}
and
\begin{displaymath}
\begin{split}
\sum_{k=1}^{n}\frac{1}{k\sqrt{k}}
&=\zeta\left(\frac{3}{2}\right)-\frac{2}{\sqrt{n}}+2\sum_{k=1}^{\infty}\frac{(-1)^{k+1}\sum_{l=1}^{k}(-1)^{l}\frac{(2l-1)!!}{2^{l}l!}B_{l}S_{k}^{(1)}(l)}{\sqrt{n}(n+1)(n+2)\cdots(n+k)}\\
&=\zeta\left(\frac{3}{2}\right)-\frac{2}{\sqrt{n}}+\frac{1}{2\sqrt{n}(n+1)}+\frac{3}{8\sqrt{n}(n+1)(n+2)}+\frac{5}{8\sqrt{n}(n+1)(n+2)(n+3)}\\
&\quad+\frac{631}{384\sqrt{n}(n+1)(n+2)(n+3)(n+4)}-\ldots.
\end{split}
\end{displaymath}}

\item[9.)]{{\bf Summation formulas for the $n$-th partial sum of $\zeta(5/2)$:}\newline
\noindent For every natural number $n\in\mathbb{N}$, we have that
\begin{displaymath}
\begin{split}
\sum_{k=1}^{n}\frac{1}{k^{2}\sqrt{k}}&=\zeta\left(\frac{5}{2}\right)-\frac{2}{3n^{\frac{3}{2}}}+\frac{1}{2n^{2}\sqrt{n}}+\frac{4}{3n\sqrt{n}}\sum_{k=1}^{\infty}\frac{(-1)^{k+1}\sum_{l=1}^{k}(-1)^{l}\frac{(2l+3)!!}{2^{l+2}(l+1)!}B_{l+1}S_{k}^{(1)}(l)}{n(n+1)(n+2)\cdots(n+k)}\\
&=\zeta\left(\frac{5}{2}\right)-\frac{2}{3n^{\frac{3}{2}}}+\frac{1}{2n^{2}\sqrt{n}}-\frac{5}{24n^{2}\sqrt{n}(n+1)}-\frac{5}{24n^{2}\sqrt{n}(n+1)(n+2)}\\
&\quad-\frac{139}{384n^{2}\sqrt{n}(n+1)(n+2)(n+3)}-\frac{59}{64n^{2}\sqrt{n}(n+1)(n+2)(n+3)(n+4)}-\ldots
\end{split}
\end{displaymath}
and that
\begin{displaymath}
\begin{split}
\sum_{k=1}^{n}\frac{1}{k^{2}\sqrt{k}}&=\zeta\left(\frac{5}{2}\right)-\frac{2}{3n^{\frac{3}{2}}}+\frac{4}{3\sqrt{n}}\sum_{k=1}^{\infty}\frac{(-1)^{k+1}\sum_{l=1}^{k}(-1)^{l}\frac{(2l+1)!!}{2^{l+1}l!}B_{l}S_{k}^{(1)}(l)}{n(n+1)(n+2)\cdots(n+k)}\\
&=\zeta\left(\frac{5}{2}\right)-\frac{2}{3n^{\frac{3}{2}}}+\frac{1}{2n\sqrt{n}(n+1)}+\frac{7}{24n\sqrt{n}(n+1)(n+2)}+\frac{3}{8n\sqrt{n}(n+1)(n+2)(n+3)}\\
&\quad+\frac{293}{384n\sqrt{n}(n+1)(n+2)(n+3)(n+4)}+\ldots.
\end{split}
\end{displaymath}}

\item[10.)]{{\bf Convergent versions of Stirling's formula:}\newline
\noindent For every natural number $n\in\mathbb{N}$, we have that
\begin{displaymath}
\begin{split}
\sum_{k=1}^{n}\log(k)&=n\log(n)-n+\frac{1}{2}\log\left(2\pi\right)+\frac{1}{2}\log(n)+\sum_{k=1}^{\infty}(-1)^{k}\frac{\sum_{l=1}^{k}\frac{(-1)^{l}}{l(l+1)}B_{l+1}S_{k}^{(1)}(l)}{(n+1)(n+2)\cdots(n+k)}\\
&=n\log(n)-n+\frac{1}{2}\log\left(2\pi\right)+\frac{1}{2}\log(n)+\frac{1}{12(n+1)}+\frac{1}{12(n+1)(n+2)}\\
&\quad+\frac{59}{360(n+1)(n+2)(n+3)}+\frac{29}{60(n+1)(n+2)(n+3)(n+4)}+\ldots,
\end{split}
\end{displaymath}
as well as
\begin{displaymath}
\begin{split}
\sum_{k=1}^{n}\log(k)&=n\log(n)-n+\frac{1}{2}\log\left(2\pi\right)+\frac{1}{2}\log(n)+\frac{1}{12n}+\sum_{k=1}^{\infty}\frac{(-1)^{k}\sum_{l=1}^{k}\frac{(-1)^{l}}{(l+1)(l+2)}B_{l+2}S_{k}^{(1)}(l)}{n(n+1)(n+2)\cdots(n+k)}\\
&=n\log(n)-n+\frac{1}{2}\log\left(2\pi\right)+\frac{1}{2}\log(n)+\frac{1}{12n}-\frac{1}{360n(n+1)(n+2)}\\
&\quad-\frac{1}{120n(n+1)(n+2)(n+3)}-\frac{5}{168n(n+1)(n+2)(n+3)(n+4)}\\
&\quad-\frac{11}{84n(n+1)(n+2)(n+3)(n+4)(n+5)}-\ldots
\end{split}
\end{displaymath}
and
\begin{displaymath}
\begin{split}
\sum_{k=1}^{n}\log(k)&=n\log(n)-n+\frac{1}{2}\log\left(2\pi\right)+\frac{1}{2}\log(n)+n\sum_{k=1}^{\infty}\frac{(-1)^{k}\sum_{l=2}^{k}\frac{(-1)^{l}}{l(l-1)}B_{l}S_{k}^{(1)}(l)
}{(n+1)(n+2)\cdots(n+k)}\\
&=n\log(n)-n+\frac{1}{2}\log\left(2\pi\right)+\frac{1}{2}\log(n)+\frac{n}{12(n+1)(n+2)}+\frac{n}{4(n+1)(n+2)(n+3)}\\
&\quad+\frac{329n}{360(n+1)(n+2)(n+3)(n+4)}+\frac{149n}{36(n+1)(n+2)(n+3)(n+4)(n+5)}+\ldots.
\end{split}
\end{displaymath}}

\item[11.)]{{\bf First logarithmic summation formulas:}\newline
\noindent For every natural number $n\in\mathbb{N}$, we have
\begin{displaymath}
\begin{split}
\sum_{k=0}^{n}k\log(k)&=\frac{1}{2}n^{2}\log(n)-\frac{1}{4}n^{2}+\frac{1}{2}n\log(n)+\frac{1}{12}\log(n)+\frac{1}{12}-\zeta^{'}(-1)\\
&\quad+\sum_{k=1}^{\infty}(-1)^{k+1}\frac{\sum_{l=1}^{k}\frac{(-1)^{l}}{l(l+1)(l+2)}B_{l+2}S_{k}^{(1)}(l)}{(n+1)(n+2)\cdots(n+k)}\\
&=\frac{1}{2}n^{2}\log(n)-\frac{1}{4}n^{2}+\frac{1}{2}n\log(n)+\frac{1}{12}\log(n)+\frac{1}{12}-\zeta^{'}(-1)+\frac{1}{720(n+1)(n+2)}\\
&\quad+\frac{1}{240(n+1)(n+2)(n+3)}+\frac{19}{1260(n+1)(n+2)(n+3)(n+4)}\\
&\quad+\frac{17}{252(n+1)(n+2)(n+3)(n+4)(n+5)}+\ldots
\end{split}
\end{displaymath}
\noindent and
\begin{displaymath}
\begin{split}
\sum_{k=0}^{n}k\log(k)&=\frac{1}{2}n^{2}\log(n)-\frac{1}{4}n^{2}+\frac{1}{2}n\log(n)+\frac{1}{12}\log(n)+\frac{1}{12}-\zeta^{'}(-1)\\
&\quad+\sum_{k=1}^{\infty}(-1)^{k+1}\frac{\sum_{l=1}^{k}\frac{(-1)^{l}}{(l+1)(l+2)(l+3)}B_{l+3}S_{k}^{(1)}(l)}{n(n+1)(n+2)\cdots(n+k)}\\
&=\frac{1}{2}n^{2}\log(n)-\frac{1}{4}n^{2}+\frac{1}{2}n\log(n)+\frac{1}{12}\log(n)+\frac{1}{12}-\zeta^{'}(-1)+\frac{1}{720n(n+1)}\\
&\quad+\frac{1}{720n(n+1)(n+2)}+\frac{13}{5040n(n+1)(n+2)(n+3)}\\
&\quad+\frac{1}{140n(n+1)(n+2)(n+3)(n+4)}+\ldots.
\end{split}
\end{displaymath}
}

\item[12.)]{{\bf Second logarithmic summation formula:}\newline
\noindent For every natural number $n\in\mathbb{N}$, we have
\begin{displaymath}
\begin{split}
\sum_{k=1}^{n}\frac{\log(k)}{k}&=\frac{1}{2}\log(n)^{2}+\frac{\log(n)}{2n}+\gamma_{1}+\sum_{k=1}^{\infty}(-1)^{k}\frac{\sum_{l=1}^{k}\frac{(-1)^{l}}{(l+1)!}B_{l+1}S_{l+1}^{(1)}(2)S_{k}^{(1)}(l)}{n(n+1)(n+2)\cdots(n+k)}\\
&\quad+\log(n)\sum_{k=1}^{\infty}\frac{(-1)^{k}\sum_{l=1}^{k}\frac{1}{l+1}B_{l+1}S_{k}^{(1)}(l)}{n(n+1)(n+2)\cdots(n+k)}\\
&=\frac{1}{2}\log(n)^{2}+\frac{\log(n)}{2n}+\gamma_{1}+\frac{1}{12n(n+1)}+\frac{1}{12n(n+1)(n+2)}\\
&\quad+\frac{109}{720n(n+1)(n+2)(n+3)}+\frac{49}{120n(n+1)(n+2)(n+3)(n+4)}+\ldots\\
&\quad-\frac{\log(n)}{12n(n+1)}-\frac{\log(n)}{12n(n+1)(n+2)}-\frac{19\log(n)}{120n(n+1)(n+2)(n+3)}\\
&\quad-\frac{9\log(n)}{20n(n+1)(n+2)(n+3)(n+4)}-\ldots.
\end{split}
\end{displaymath}}

\newpage

\item[13.)]{{\bf Third logarithmic summation formula (Summation formula for the $n$-th partial sum of $\zeta'(2)$):}\newline
\noindent For every natural number $n\in\mathbb{N}$, we have
\begin{displaymath}
\begin{split}
\sum_{k=1}^{n}\frac{\log(k)}{k^{2}}&=-\zeta^{'}(2)-\frac{\log(n)}{n}-\frac{1}{n}+\frac{\log(n)}{2n^{2}}+\frac{1}{n}\sum_{k=1}^{\infty}(-1)^{k+1}\frac{\sum_{l=1}^{k}\frac{1}{l+1}\left(\sum_{m=0}^{l-1}\frac{m+1}{l-m}\right)B_{l+1}S_{k}^{(1)}(l)}{n(n+1)(n+1)\cdots(n+k)}\\
&\quad+\frac{\log(n)}{n}\sum_{k=1}^{\infty}\frac{(-1)^{k}\sum_{l=1}^{k}B_{l+1}S_{k}^{(1)}(l)}{n(n+1)(n+2)\cdots(n+k)}\\
&=-\zeta^{'}(2)-\frac{\log(n)}{n}-\frac{1}{n}+\frac{\log(n)}{2n^{2}}+\frac{1}{12n^{2}(n+1)}+\frac{1}{12n^{2}(n+1)(n+2)}\\
&\quad+\frac{47}{360n^{2}(n+1)(n+2)(n+3)}+\frac{17}{60n^{2}(n+1)(n+2)(n+3)(n+4)}+\ldots\\
&\quad-\frac{\log(n)}{6n^{2}(n+1)}-\frac{\log(n)}{6n^{2}(n+1)(n+2)}-\frac{3\log(n)}{10n^{2}(n+1)(n+2)(n+3)}\\
&\quad-\frac{4\log(n)}{5n^{2}(n+1)(n+2)(n+3)(n+4)}-\ldots.
\end{split}
\end{displaymath}}

\item[14.)]{{\bf Fourth logarithmic summation formula:}\newline
\noindent For every natural number $n\in\mathbb{N}$, we have
\begin{displaymath}
\begin{split}
\sum_{k=1}^{n}\log(k)^2&=n\log(n)^{2}-2n\log(n)+2n+\frac{1}{2}\log(n)^{2}+\frac{\log(n)}{6n}+\frac{\gamma^{2}}{2}-\frac{\pi^{2}}{24}-\frac{\log(2)^{2}}{2}\\
&\quad-\log(2)\log(\pi)-\frac{\log(\pi)^{2}}{2}+\gamma_{1}+2\sum_{k=1}^{\infty}\frac{(-1)^{k}\sum_{l=1}^{k}\frac{(-1)^{l}}{(l+2)!}B_{l+2}S_{l+1}^{(1)}(2)S_{k}^{(1)}(l)}{n(n+1)(n+2)\cdots(n+k)}\\
&\quad+2\log(n)\sum_{k=1}^{\infty}\frac{(-1)^{k}\sum_{l=1}^{k}\frac{1}{(l+1)(l+2)}B_{l+2}S_{k}^{(1)}(l)}{n(n+1)(n+2)\cdots(n+k)}\\
&=n\log(n)^{2}-2n\log(n)+2n+\frac{1}{2}\log(n)^{2}+\frac{\log(n)}{6n}+\frac{\gamma^{2}}{2}-\frac{\pi^{2}}{24}-\frac{\log(2)^{2}}{2}\\
&\quad-\log(2)\log(\pi)-\frac{\log(\pi)^{2}}{2}+\gamma_{1}+\frac{1}{120n(n+1)(n+2)}\\
&\quad+\frac{1}{40n(n+1)(n+2)(n+3)}+\frac{167}{1890n(n+1)(n+2)(n+3)(n+4)}\\
&\quad+\frac{145}{378n(n+1)(n+2)(n+3)(n+4)(n+5)}+\ldots\\
&\quad-\frac{\log(n)}{180n(n+1)(n+2)}-\frac{\log(n)}{60n(n+1)(n+2)(n+3)}\\
&\quad-\frac{5\log(n)}{84n(n+1)(n+2)(n+3)(n+4)}-\frac{11\log(n)}{42n(n+1)(n+2)(n+3)(n+4)(n+5)}-\ldots.
\end{split}
\end{displaymath}}

\item[15.)]{{\bf Summation formulas for the $n$-th partial sum of the Gregory-Leibniz series:}\newline
\noindent For every natural number $n\in\mathbb{N}_{0}$, we have that
\begin{displaymath}
\begin{split}
\sum_{k=0}^{n}\frac{(-1)^k}{2k+1}&=\frac{\pi}{4}-\frac{(-1)^{n+1}}{4(n+1)}+\frac{(-1)^{n+1}}{4}\sum_{k=1}^{\infty}\frac{(-1)^{k+1}\sum_{l=1}^{k}\frac{(-1)^{l}}{2^{l}}E_{l}S_{k}^{(1)}(l)}{(n+1)(n+2)\cdots(n+k+1)}\\
&=\frac{\pi}{4}-\frac{(-1)^{n+1}}{4(n+1)}+\frac{(-1)^{n+1}}{16(n+1)(n+2)(n+3)}+\frac{3(-1)^{n+1}}{16(n+1)(n+2)(n+3)(n+4)}\\
&\quad+\frac{39(-1)^{n+1}}{64(n+1)(n+2)(n+3)(n+4)(n+5)}\\
&\quad+\frac{75(-1)^{n+1}}{32(n+1)(n+2)(n+3)(n+4)(n+5)(n+6)}\\
&\quad+\ldots
\end{split}
\end{displaymath}
and for all $n\in\mathbb{N}$ that
\begin{displaymath}
\begin{split}
\sum_{k=1}^{n}\frac{(-1)^{k+1}}{2k-1}&=\frac{\pi}{4}-\frac{(-1)^{n}}{4n}+\frac{(-1)^{n}}{4}\sum_{k=1}^{\infty}\frac{(-1)^{k+1}\sum_{l=1}^{k}\frac{(-1)^{l}}{2^{l}}E_{l}S_{k}^{(1)}(l)}{n(n+1)(n+2)\cdots(n+k)}\\
&=\frac{\pi}{4}-\frac{(-1)^{n}}{4n}+\frac{(-1)^{n}}{16n(n+1)(n+2)}+\frac{3(-1)^{n}}{16n(n+1)(n+2)(n+3)}\\
&\quad+\frac{39(-1)^{n}}{64n(n+1)(n+2)(n+3)(n+4)}+\frac{75(-1)^{n}}{32n(n+1)(n+2)(n+3)(n+4)(n+5)}+\ldots.
\end{split}
\end{displaymath}}

\item[16.)]{{\bf Summation formula for the $n$-th partial sum of the alternating harmonic series:}\newline
\noindent For every natural number $n\in\mathbb{N}$, we have\begin{displaymath}
\begin{split}
\sum_{k=1}^{n}\frac{(-1)^{k+1}}{k}&=\log(2)-\frac{(-1)^{n}}{2n}+(-1)^{n}\sum_{k=1}^{\infty}(-1)^{k}\frac{\sum_{l=1}^{k}(-1)^{\frac{l^2+l}{2}}\frac{\left(2^{l+1}-1\right)}{l+1}\left|B_{l+1}\right|S_{k}^{(1)}(l)}{n(n+1)(n+2)\cdots(n+k)}\\
&=\log(2)-\frac{(-1)^{n}}{2n}+\frac{(-1)^{n}}{4n(n+1)}+\frac{(-1)^{n}}{4n(n+1)(n+2)}+\frac{3(-1)^{n}}{8n(n+1)(n+2)(n+3)}\\
&\quad+\frac{3(-1)^{n}}{4n(n+1)(n+2)(n+3)(n+4)}+\ldots.
\end{split}
\end{displaymath}
}
\end{itemize}

\noindent To obtain the formulas $15.)$ and $16.)$, we have used the references \cite{13,18,19,20} and the Boole summation formula \cite{13,18}, which is the analog of the Euler-Maclaurin summation formula for alternating sums.

\section{Conclusion}
\label{sec:Conclusion}

We have proved many rapidly convergent summation formulas for different sums of functions. These formulas are highly useful, because of their fast convergence and their connections to many different special functions, like for example the polygamma functions. Perhaps one can implement these formulas in a future version of Mathematica to compute for example the digamma function via the formula
\begin{displaymath}
\psi(x)=\log(x)-\frac{1}{2x}+\sum_{k=1}^{\infty}\frac{(-1)^{k+1}\sum_{l=1}^{k}\frac{(-1)^{l}}{l+1}B_{l+1}S_{k}^{(1)}(l)}{x(x+1)(x+2)\cdots(x+k)}.
\end{displaymath}
Computing for example $\psi(10^{10})$ up to $400$ correct decimal digits with the current version of Mathematica $10$ takes more than $10$ minutes on our home computer, but with the above formula, we can do it in less than $3$ seconds.\newline
\noindent With our technique, one can obtain much more such summation formulas for finite sums of the form $\sum_{k=1}^{n}f(k)$, where $n\in\mathbb{N}$ and $f$ is any function, showing again that this universal technique for obtaining such formulas is highly suitable for computer implementations.

\bigskip
\hrule
\bigskip

\noindent 2010 {\it Mathematics Subject Classification}: Primary 40G99; Secondary 40A05.

\noindent\emph{Keywords:} Rapidly convergent summation formulas for finite sums of functions involving Stirling series; Inverse factorial series; Stirling series; Stirling summation formulas; Convergent version of the Euler-Maclaurin summation formula; Convergent version of the Boole summation formula; Weniger transformation; Stirling numbers of the first kind; Stirling numbers of the second kind; Bernoulli numbers; Euler numbers; tangent numbers; Gregory numbers; Convergent versions of Stirling's formula; Convergent formulas for the harmonic series; Gregory-Leibniz series; alternating harmonic series.

\end{document}